\documentclass[a4paper]{article}
\usepackage{lmodern}
\pdfoutput=1
\usepackage{graphicx,wrapfig}
\usepackage{lipsum}
\usepackage{amsmath,amssymb,mathrsfs}
\usepackage{hyperref}
\usepackage{url}
\usepackage{mathtools}
\usepackage{changepage}
\usepackage[ruled,vlined]{algorithm2e}
\usepackage{amsmath}
\usepackage{multirow}
\usepackage{wrapfig}
\usepackage{subfig}
\usepackage{color}
\usepackage{epsfig,enumerate,amsmath,amsfonts,amssymb,amsthm,mathrsfs,ifpdf}
\usepackage{indentfirst,relsize}
\usepackage{setspace,graphicx}
\usepackage{latexsym}
\usepackage[all]{xy}
\usepackage[dvipsnames]{pstricks}
\usepackage{pst-grad} 
\usepackage{pst-plot} 
\usepackage[margin = 2.50cm]{geometry}

\usepackage{authblk} 
\usepackage{algpseudocode}
\usepackage{color,soul}
\usepackage{enumitem}

\usepackage{orcidlink}

\newcommand{\remove}[1]{}

\newtheorem{theorem}{Theorem}
\newtheorem{lemma}[theorem]{Lemma}

\newtheorem{definition}[theorem]{Definition}

\newtheorem{observation}[theorem]{Observation}

\newcommand{\DP}{\textsc{DP}}
\newcommand{\eop}{\textsc{Maximum Edge Open Packing}}
\newcommand{\maxeop}{\textsc{MaxEOP}}
\usepackage{todonotes}
\usepackage[most]{tcolorbox}

\title{Algorithms for the \eop~Problem}

\author[1]{Sriram Bhyravarapu\footnote{sriramb@iitg.ac.in}}
\author[2]{Gautam K. Das\footnote{gkd@iitg.ac.in}}
\author[3]{Kamal Santra \orcidlink{0009-0006-5997-1452} \footnote{kamal.7.2013@gmail.com, kamal.santra@iitg.ac.in}}
\affil[1]{Department of Computer Science and Engineering\\
	
	Indian Institute of Technology Guwahati \\
	
	Guwahati, 781039, Assam, India}

\affil[2, 3]{Department of Mathematics\\
	
	Indian Institute of Technology Guwahati\\
	
	Guwahati, 781039, Assam, India}
\date{}

\begin{document}

	\maketitle
	\begin{abstract}
		Packing problems form a central theme in graph theory, owing to their relevance in 
		modeling conflict-free resource allocation, network design, and communication 
		constraints. Motivated by applications in wireless networks where each device can 
		participate in at most one communication at a time and simultaneous links must 
		avoid interference we consider a generalization of induced matching known as 
		\emph{edge open packing}. Two edges of a graph are said to conflict if a third 
		edge connects one endpoint of each; an \emph{edge open packing set} is a set of 
		edges containing no such conflicting pair. The largest cardinality of such a set 
		is the \emph{edge open packing number} of a graph.
		
		In this work, we study the computational complexity of the Maximum Edge Open 
		Packing Problem. We give a polynomial-time algorithm for the problem in 
		\emph{distance-hereditary graphs}, exploiting their canonical decomposition via 
		twin-set interactions. We further show that the problem remains polynomial-time 
		solvable on \emph{biconvex bipartite graphs}, thereby identifying a tractable 
		subclass within bipartite graphs, in contrast to the known NP-hardness of the 
		problem on Eulerian bipartite graphs. Finally, we initiate the parameterized 
		complexity study of the problem and present a fixed-parameter tractable algorithm 
		for \emph{chordal graphs}, parameterized by the clique number $\omega$, running 
		in $O(2^{\omega}\cdot\mathrm{poly}(n))$ time.
	\end{abstract}


	\section{Introduction}
Many fundamental optimization problems in graph theory arise from the need to 
identify collections of edges or vertices that behave in an independent or 
non-interfering manner. Classical examples include matchings, induced matchings, 
packings, and colorings, all of which capture different forms of combinatorial 
separation within a graph. Such structures play an important role in applications 
ranging from scheduling and resource allocation to communication networks and 
distributed computing, where conflicts between chosen elements must be avoided \cite{cameron1989induced, cardoso2019injective, dabrowski2013new, lozin2002maximum, stockmeyer1982np}


One well-studied notion in this context is the \emph{induced matching}, a set of 
pairwise non-adjacent edges whose endpoints are also mutually non-adjacent. 
Induced matchings provide a strict form of interference avoidance and have been 
investigated extensively due to their algorithmic complexity and diverse
applications. However, in many settings such as communication networks edges may
be required to satisfy a weaker form of separation: rather than forbidding 
adjacency between their endpoints, it may be sufficient to ensure that no 
\emph{third} edge connects two chosen edges. This relaxation motivates the study 
of \emph{edge open packings}, introduced by Chelladurai et al.~\cite{chelladurai2022edge}.

Given a graph $G=(V,E)$, a set $D \subseteq E$ is an \emph{edge open packing} 
(EOP) if no two edges in $D$ are linked by an intermediate edge of $G$. 
Equivalently, for any distinct $e_1,e_2 \in D$, the edge-induced subgraph 
$G\langle e_1,e_2\rangle$ extends no further to form a $P_4$ or a triangle when 
combined with any other edge of $G$. The maximum size of such a set, denoted 
$\rho^o_e(G)$, is the \emph{edge open packing number}. In contrast to induced 
matchings, whose induced subgraph consists of isolated $K_{1,1}$ components, an 
EOP set may induce larger stars, enabling a richer structural behaviour. This 
makes the Maximum Edge Open Packing Problem (MaxEOP) a natural and flexible 
generalization of induced matching.

Recent work has revealed that determining $\rho^o_e(G)$ is computationally 
challenging in general. Bre\v{s}ar and Samadi~\cite{brevsar2024edge} proved that 
MaxEOP is NP-complete even on several restricted graph classes, including planar 
graphs of maximum degree~$4$ and Eulerian bipartite graphs, while also giving a 
linear-time algorithm for trees. Subsequent research deepened the structural 
understanding of EOP sets and explored their relationships with induced matchings 
and graph parameters such as diameter, minimum degree, and clique number 
\cite{BRESAR2025, chelladurai2022edge, pandey2025edge,eopblock}.

Despite this progress, the algorithmic landscape for MaxEOP remains only partially 
understood. In particular, the complexity of the problem for a broad range of 
graph classes with strong structural decompositions has not yet been completely 
resolved. This motivates our investigation of MaxEOP in graph classes where 
recursive constructions and decomposition trees can be exploited to design
efficient algorithms.

\medskip
\noindent
\textbf{Our contributions.}
In this work, we make progress toward this direction by presenting new algorithms 
for MaxEOP in three well-studied graph classes:
\begin{itemize}
	\item In Section \ref{EOP_Sec_2}, we give a polynomial-time algorithm for MaxEOP in 
	\emph{distance-hereditary graphs}, using the twin-set based decomposition 
	introduced by Chang et al.~\cite{chang1997dynamic}. Our dynamic programming 
	method systematically computes EOP values along the decomposition tree. This result generalizes the result on block graphs from \cite{eopblock}. 
	
	\item Because MaxEOP is NP-hard even on Eulerian bipartite graphs \cite{brevsar2024edge}, identifying 
	tractable subclasses becomes crucial. We show that 
	the problem admits a polynomial-time algorithm 
	on the more restricted class of biconvex bipartite graphs, in Section \ref{EOP_Sec_3}.

	\item  A question raised by Bre\v{s}ar and Samadi~\cite{brevsar2024edge} asks for the 
	computational complexity of MaxEOP on \emph{chordal graphs}. 
	We provide a partial affirmative answer by designing an 
	$O(2^{\omega}\cdot \mathrm{poly}(n))$ fixed-parameter tractable algorithm 
	parameterized by the clique number~$\omega$, in Section \ref{EOP_Sec_4}. As a consequence, the problem becomes 
	polynomial-time solvable for chordal graphs whose clique number is constant. 
\end{itemize}
Finally, we conclude with an open question in Section \ref{EOP_Sec_5}.

\section{Preliminaries}
In this section, we introduce the notation and graph classes used throughout the paper. For terminology not defined here, we refer the reader to the standard text of Diestel~\cite{diestel2005graph}. 

Let $G=(V,E)$ be a graph. We use $V(G)$ and $E(G)$ to denote the vertex set and edge set of $G$, respectively, and omit the argument whenever the graph is clear from the context. The \emph{order} of $G$ is the number of vertices in $G$, namely $|V(G)|$. For a vertex $v\in V(G)$, the \emph{open neighborhood} of $v$ is denoted by $N_G(v)$, while the \emph{closed neighborhood} is denoted by $N_G[v]=N_G(v)\cup\{v\}$. When no confusion arises, we simply write $N(v)$ and $N[v]$. The \emph{degree} of a vertex $v$ is $\deg(v)=|N(v)|$, and the maximum degree of $G$ is denoted by $\Delta(G)$, or simply $\Delta$.

For a subset $U\subseteq V(G)$, we denote by $\deg_U(v)$ the number of neighbors of $v$ contained in $U$, and by $N_U(v)$ the set of neighbors of $v$ in $U$. A vertex is called a \emph{pendant vertex} (or \emph{leaf}) if it has degree one, and an \emph{isolated vertex} if it has degree zero. Given a subset $S\subseteq V(G)$, the subgraph induced by $S$ is denoted by $G[S]$. Two adjacent vertices $u$ and $v$ are called \emph{true twins} if $N_G[u]=N_G[v]$, whereas two non-adjacent vertices are called \emph{false twins} if $N_G(u)=N_G(v)$. We denote by $P_n$ the path on $n$ vertices.

We now recall the graph classes that are relevant to our work. A graph $G$ is called \emph{distance-hereditary} if, for every pair of vertices, their distance remains unchanged in every connected induced subgraph containing them. Distance-hereditary graphs were introduced by Howorka~\cite{howorka1977characterization} and have been extensively studied due to their rich structural properties and algorithmic applications~\cite{d1988distance,hammer1990completely,hsieh2002characterization,yeh1995weighted}. Hammer and Maffray~\cite{hammer1990completely} showed that distance-hereditary graphs can be recognized in linear time.

A graph $G=(V,E)$ is said to be \emph{bipartite} if its vertex set can be partitioned into two disjoint independent sets $X$ and $Y$ such that every edge of $G$ has one endpoint in $X$ and the other in $Y$. In such a case, we write $G=(X\cup Y,E)$. A bipartite graph $G=(X\cup Y,E)$ is called a \emph{bipartite chain graph} if there exist orderings
$
\sigma_X=(x_1,x_2,\ldots,x_{n_1})
\quad\text{and}\quad
\sigma_Y=(y_1,y_2,\ldots,y_{n_2})
$
of the vertices of $X$ and $Y$, respectively, such that
$
N(x_{n_1})\subseteq N(x_{n_1-1})\subseteq \cdots \subseteq N(x_1)
$
and
$
N(y_{n_2})\subseteq N(y_{n_2-1})\subseteq \cdots \subseteq N(y_1).
$
Such orderings are called \emph{chain orderings}. It is known that chain orderings can be computed in linear time~\cite{heggernes2007linear}. 

A bipartite graph $G=(X\cup Y,E)$ is called \emph{biconvex bipartite} if the vertices of $X$ and $Y$ can be ordered in such a way that, for every vertex $y\in Y$, the neighbors of $y$ appear consecutively in the ordering of $X$, and similarly, for every vertex $x\in X$, the neighbors of $x$ appear consecutively in the ordering of $Y$. In this case, both parts satisfy the \emph{convexity property}. The following equivalent characterization of connected biconvex bipartite graphs will be useful later.

\begin{definition}[Biconvex Bipartite Graph \cite{BHYRAVARAPU2025115260, Diaz2021}]\label{def:biconv}
	A connected graph $G=(V,E)$ is biconvex bipartite
	if and only if 
	$V(G)$ can be partitioned into $p+1$ disjoint independent sets $L_0,L_1,\dots, L_p$ (in this order) in such a way that $|L_0|=1$, any two vertices in non-consecutive sets are non-adjacent, and 
	\begin{enumerate}
		\item Any two consecutive sets $L_{i-1}$ and $L_{i}$ induce a chain graph, denoted by $G_i$. 
		\item For each $i\in \{1, 2, \dots, p-1\}$, there are two orderings $\sigma_{i, 1}$ and $\sigma_{i,2}$ 
		of vertices of the set $L_i$ 
		such that 
		$\sigma_{i,1}$ is non-increasing in $G_i$ and 
		$\sigma_{i,2}$ is non-decreasing in $G_{i+1}$. 
		For the set $L_0$ (resp. $L_p$), there is a non-decreasing (resp. non-increasing) ordering of vertices of $L_0$ (resp. $L_p$) in $G_1$ (resp. $G_p$). 
		
	\end{enumerate}
\end{definition}

Finally, we briefly recall the notions of tree decomposition and nice tree
decomposition, which are useful in many dynamic programming algorithms on
graphs. A \emph{tree decomposition} of a graph \(G=(V,E)\) is a pair
\((T,\mathcal{B})\), where \(T\) is a tree and
\(\mathcal{B}=\{B_t\subseteq V(G):t\in V(T)\}\) is a collection of subsets of
\(V(G)\), called \emph{bags}, satisfying the following properties:
\begin{enumerate}
	\item \(\bigcup_{t\in V(T)}B_t=V(G)\);
	
	\item for every edge \(uv\in E(G)\), there exists a node \(t\in V(T)\) such
	that \(\{u,v\}\subseteq B_t\); and
	
	\item for every vertex \(v\in V(G)\), the set
	\(\{t\in V(T):v\in B_t\}\) induces a subtree of \(T\).
\end{enumerate}

The \emph{width} of a tree decomposition is
\(\max_{t\in V(T)}|B_t|-1\), and the \emph{treewidth} of \(G\) is the minimum
width over all tree decompositions of \(G\).

To illustrate the concept, consider the path graph \(P_4\) with vertex set
\(\{a,b,c,d\}\). A tree decomposition of width \(1\) is obtained by taking
three bags \(\{a,b\}\), \(\{b,c\}\), and \(\{c,d\}\), arranged in a path. The
corresponding decomposition is shown in Figure~\ref{fig:tree-dec-example}.

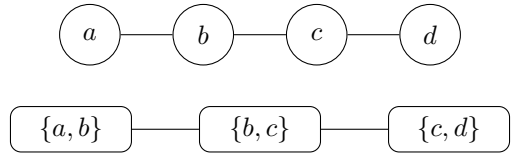
\begin{figure}[h]
	\centering
	
	\begin{tikzpicture}[every node/.style={circle, draw, minimum size=8mm}]
		\node (a) at (0,0) {$a$};
		\node (b) at (1.5,0) {$b$};
		\node (c) at (3,0) {$c$};
		\node (d) at (4.5,0) {$d$};
		
		\draw (a)--(b)--(c)--(d);
	\end{tikzpicture}
	
	\vspace{0.5cm}
	
	\begin{tikzpicture}[every node/.style={draw, rectangle, rounded corners, minimum width=16mm, minimum height=6mm}]
		\node (t1) at (0,0) {$\{a,b\}$};
		\node (t2) at (2.5,0) {$\{b,c\}$};
		\node (t3) at (5,0) {$\{c,d\}$};
		
		\draw (t1)--(t2)--(t3);
	\end{tikzpicture}
	
	\caption{A path graph together with a tree decomposition of width \(1\).}
	\label{fig:tree-dec-example}
\end{figure}

A \emph{nice tree decomposition} is a rooted tree decomposition in which every
node is one of the following four types:
\begin{enumerate}
	\item a \emph{leaf node} with an empty bag;
	
	\item an \emph{introduce node} that introduces exactly one vertex into the
	bag of its child;
	
	\item a \emph{forget node} that removes exactly one vertex from the bag of
	its child; or
	
	\item a \emph{join node} having two children with identical bags.
\end{enumerate}

\section{Distance-hereditary graphs}\label{EOP_Sec_2}

In this section, we develop a polynomial time dynamic programming based algorithm for computing a maximum edge open packing set of a distance-hereditary graph $G$. Towards this, we make use of a decomposition tree. 

Unlike the one-vertex-extension characterization, Chang et al.~\cite{chang1997dynamic} 
study distance-hereditary graphs through the interaction structure of edges between 
two designated vertex sets, called \emph{twin sets}. A graph $G$ consisting of a 
single vertex $v$ is regarded as a distance-hereditary graph with twin set $TS(G)=\{v\}$. Let $G_l$ and $G_r$ be two distance-hereditary graphs with twin sets 
$TS(G_l)$ and $TS(G_r)$, respectively. A new distance-hereditary graph $G$ can be 
constructed from $G_l$ and $G_r$ by applying one of the following three fundamental 
operations: the \emph{true twin} operation, the \emph{false twin} operation, or 
the \emph{attachment} operation. These operations form the basis for a canonical 
decomposition tree representation of distance-hereditary graphs; see 
Figure~\ref{fig:Decomposition_tree}. The three operations are defined as follows. 

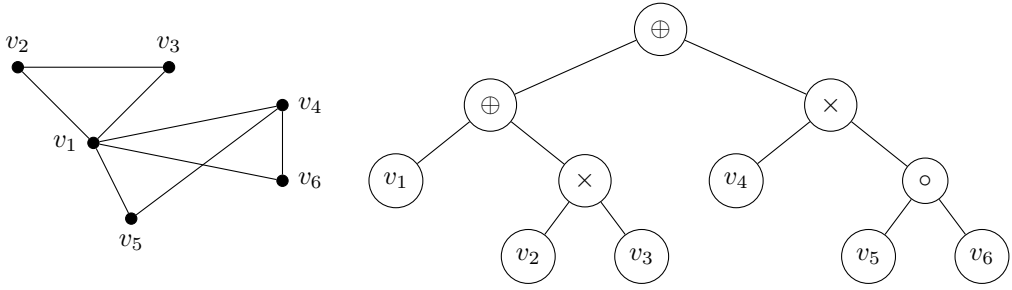
\begin{figure}[ht]
	\centering
	\begin{tikzpicture}[
		vertex/.style={circle, draw, fill=black, inner sep=1.5pt},
		leaf/.style={circle, draw, minimum size=6mm},
		op/.style={circle, draw, minimum size=6mm},
		level 1/.style={sibling distance=45mm},
		level 2/.style={sibling distance=25mm},
		level 3/.style={sibling distance=15mm},
		level distance=10mm
		]
		
		
		\node[vertex,label=left:$v_1$] (v1) at (0,0) {};
		\node[vertex,label=above:$v_2$] (v2) at (-1,1) {};
		\node[vertex,label=above:$v_3$] (v3) at (1,1) {};
		\node[vertex,label=right:$v_4$] (v4) at (2.5,0.5) {};
		\node[vertex,label=below:$v_5$] (v5) at (0.5,-1) {};
		\node[vertex,label=right:$v_6$] (v6) at (2.5,-0.5) {};
		
		\draw (v1)--(v2);
		\draw (v1)--(v3);
		\draw (v2)--(v3);
		\draw (v1)--(v4);
		\draw (v1)--(v5);
		\draw (v4)--(v6);
		\draw (v5)--(v4);
		\draw (v1)--(v6);   
		
		
		\node[op] (root) at (7.5,1.5) {$\oplus$}
		child {
			node[op] {$\oplus$}
			child { node[leaf] {$v_1$} }
			child {
				node[op] {$\times$}
				child { node[leaf] {$v_2$} }
				child { node[leaf] {$v_3$} }
			}
		}
		child {
			node[op] {$\times$}
			child { node[leaf] {$v_4$} }
			child {
				node[op] {$\circ$}
				child { node[leaf] {$v_5$} }
				child { node[leaf] {$v_6$} }
			}
		};
		
	\end{tikzpicture}
	
	\caption{A distance-hereditary graph (left) and its decomposition tree (right).}
	\label{fig:Decomposition_tree}
\end{figure}

\begin{enumerate}[label=\textup{(\roman*)}]
	\item \emph{True twin operation}: \(G=G_l\otimes G_r\), where
	\(V(G)=V(G_l)\cup V(G_r)\),
	\(E(G)=E(G_l)\cup E(G_r)\cup\{xy:x\in TS(G_l),\,y\in TS(G_r)\}\), and
	\(TS(G)=TS(G_l)\cup TS(G_r)\).
	
	\item \emph{False twin operation}: \(G=G_l\odot G_r\), where
	\(V(G)=V(G_l)\cup V(G_r)\), \(E(G)=E(G_l)\cup E(G_r)\), and
	\(TS(G)=TS(G_l)\cup TS(G_r)\).
	
	\item \emph{Attachment operation}: \(G=G_l\oplus G_r\), where
	\(V(G)=V(G_l)\cup V(G_r)\),
	\(E(G)=E(G_l)\cup E(G_r)\cup\{xy:x\in TS(G_l),\,y\in TS(G_r)\}\), and
	\(TS(G)=TS(G_l)\).
\end{enumerate}

Assume that a distance-hereditary graph $G$ is obtained through a sequence of the operations defined above. This construction can be encoded by a full binary tree $T$, called the \emph{decomposition tree}. The leaves of $T$ correspond exactly to the vertices of $G$, while each internal node of $T$ is labeled by one of the symbols $\otimes$, $\odot$, or $\oplus$, indicating a true twin, false twin, or attachment operation, respectively. If \(v\) is a node of \(T\), then \(L(v)\) denotes
the set of leaves in the subtree rooted at \(v\), and \(G_v\) denotes the graph
induced by \(L(v)\).

In this representation, each leaf corresponds to the distance-hereditary graph $G[\{v_i\}]$ with twin set $TS(G[\{v_i\}]) = \{v_i\}$ for $1 \le i \le n$. Each rooted subtree encodes the distance-hereditary graph whose vertex set and construction steps are determined by the subtree, and each internal node specifies the operation applied to the two distance-hereditary graphs associated with its left and right subtrees. Figure~\ref{fig:Decomposition_tree} presents an illustrative example. In particular, the right side of Figure~\ref{fig:Decomposition_tree} displays the decomposition tree $T$ associated with the distance-hereditary graph $G$ on the left of Figure~\ref{fig:Decomposition_tree}.

\begin{lemma}[\cite{chang1997dynamic,hsieh2002characterization}]
  Given a distance-hereditary graph G, a decomposition tree T of G can be generated in $O(n+m)$ time.  
\end{lemma}\label{lem:dh-decomposition-tree}


\subsection{Families of partial solutions}
We employ a dynamic programming strategy to compute a maximum edge open packing set of a graph $G$. The algorithm begins by constructing a decomposition tree $T$ of $G$, rooted at $r$. It then processes the nodes of $T$ in a bottom-up manner, computing edge open packing (EOP) sets of the subgraphs $G[L(v)]$ for every $v \in T$. Since $G[L(r)] = G$, an EOP set of $G$ is obtained at the end of this procedure. Throughout this section, we use $E(TS(G))$ to denote the edges of the induced subgraph $G[TS(G)]$. 

Let $G$ be a distance-hereditary graph and let $TS(G)$ denote its twin set. For an arbitrary EOP set $S$ of $G$, exactly one of the following cases holds:
\begin{enumerate}[]
    \item The twin set contains at least one edge of $S$;
    \item The twin set contains no edge of $S$, but contains at least one endpoint of some edge in $S$;
    \item The twin set contains neither an edge of $S$ nor an endpoint of any edge of $S$.
\end{enumerate}

Accordingly, we classify EOP sets of $G$ into the families $P_1(G)$, $P_2(G)$, and $P_3(G)$. Let $\rho_i(G)$ denote the maximum cardinality of an EOP set in $P_i(G)$. Clearly, 
$\rho_e^{o}(G) = \max\{\rho_1(G), \rho_2(G), \rho_3(G)\}$. To compute the values of $\rho_i(G)$, for each $i$, we introduce additional types of EOP sets. For each vertex $x \in TS(G)$, let $B_x$ denote the subset of vertices in $N_{G\setminus TS(G)}(x)$ such that either $B_x = \emptyset$ or $B_x$ is an independent set. Define 
$t = \max_{x \in TS(G)} |B_x|$. 
Additionally, let $A$ be an independent subset of $TS(G)$. These definitions allow us to introduce the refined families of EOP sets required for the design of the algorithm.

\paragraph*{\textbf{Definitions of EOP Set Families}}

We define the following families of edge open packing sets of \(G\). Let
\(P(G)=\{\,S:S\text{ is an EOP set in }G\,\}\). Let
\(P_1(G)=\{\,S\in P(G):S\cap E(TS(G))\neq\emptyset\,\}\). Thus \(P_1(G)\)
contains all EOP sets that select at least one edge inside the twin set.
Let \(P_2(G)=\{\,S\in P(G)\setminus P_1(G):\exists\,e=uv\in S\text{ such that }
|\{u,v\}\cap TS(G)|=1\,\}\). Thus \(P_2(G)\) contains all EOP sets that do not
select an internal twin-set edge, but select at least one edge having exactly
one endpoint in \(TS(G)\). Finally, let
\(P_3(G)=P(G)\setminus(P_1(G)\cup P_2(G))\). Hence,  \(P_3(G)\) contains all EOP
sets that avoid the twin set completely.

Let us  define
\(P_4(i,G)=\{\,S\in P_3(G):\exists\,A\subseteq TS(G)\text{ such that }A\text{
	is independent}, |A|=i,\text{ and }V(S)\cap N(A)=\emptyset\,\}\) for every \(1\le i\le \alpha(G[TS(G)])\). Similarly,
for every \(0\le j\le t\), define
\(P_5(j,G)=\{\,S\in P_3(G):\exists\,B_x\text{ with }|B_x|=j,\,
V(S)\cap N_{G\setminus TS(G)}(x)=\emptyset,\text{ and }
V(S)\cap N_G(B_x)=\emptyset\,\}\).

%
%
%
%
%
%
%
%
%
%
%
%

We now present an intuitive explanation of each of the EOP set families defined above. The family $P(G)$ consists of all edge open packing sets of $G$. 
Within this collection, the set $P_1(G)$ contains all EOP sets that select at least one edge entirely inside the twin-set $TS(G)$; such packings directly use the internal structure of the twin-set. 
The family $P_2(G)$ contains those EOP sets that do not include an internal twin edge, but use an edge having exactly one endpoint in $TS(G)$, meaning the packing interacts with the twin-set only through boundary edges. 
The class $P_3(G)$ consists of all EOP sets that completely avoid the twin-set, selecting no edge that lies in $TS(G)$ and no edge incident to a vertex of $TS(G)$. 
The family $P_4(i,G)$ refines $P_3(G)$ by restricting attention to packings that also avoid the open neighbourhood of an independent set $A \subseteq TS(G)$ of size $i$, thereby avoiding both the twin-set and all vertices adjacent to these $i$ twin vertices. 
Finally, the family $P_5(j,G)$ provides a second refinement of $P_3(G)$ by considering EOP sets that avoid the neighbourhood of an independent set $B_x \subseteq N_{G\setminus TS(G)}(x)$ of size $j$ for some $x \in TS(G)$; such packings do not select edges near $x$, the vertices of $B_x$, or any of their neighbours. 







We now define the corresponding optimal values for the above defined sets.  
%
%
%
%
%
%
%

%
%
%
%
%
%

\begin{itemize}
	\item \(\rho_e^{o}(G)=\max\{\,|S|:S\in P(G)\,\}\).
	
	\item For \(1\le i\le 3\), let
	\(\rho_i(G)=\max\{\,|S|:S\in P_i(G)\,\}\).
	
	\item For each index \(i\), let
	\(\rho_4(i,G)=\max\{\,|S|:S\in P_4(i,G)\,\}\). Moreover,
	\(\rho_4(G)=\max_i \rho_4(i,G)\).
	
	\item For each index \(j\), let
	\(\rho_5(j,G)=\max\{\,|S|:S\in P_5(j,G)\,\}\). Moreover,
	\(\rho_5(G)=\displaystyle\max_{x\in TS(G)}\max_j \rho_5(j,G)\).
\end{itemize}

Let $p_l$ and $p_r$ be the maximum integers such that 
$\rho_4(G_l)=\rho_4(p_l, G_l)$ and $\rho_4(G_r)=\rho_4(p_r, G_r)$, respectively. 
Similarly, let $q_l$ and $q_r$ be the maximum integers such that 
$\rho_5(G_l)=\rho_5(q_l, G_l)$ and $\rho_5(G_r)=\rho_5(q_r, G_r)$. 
Based on the above discussion, we obtain the following observation which is trivial to see and therefore we omit the proof details. 

\begin{observation}\label{Obs: dh_graphs_1}
Let $G$ be a connected distance-hereditary graph. Then the maximum edge open packing number, $\rho_e^o(G) = \max\{\rho_1(G),\, \rho_2(G),\, \rho_3(G)\}$. 
\end{observation}

\begin{proof}
	Every edge open packing set of \(G\) falls into exactly one of the three
	classes \(P_1(G)\), \(P_2(G)\), and \(P_3(G)\), according to whether it uses an
	edge inside \(TS(G)\), uses no such edge but has an edge incident with
	\(TS(G)\), or avoids \(TS(G)\) completely. Therefore the maximum value over all
	edge open packing sets is exactly the maximum of these three quantities.
\end{proof}

\subsection{Recurrences for the decomposition operations}

We now give the recurrences used by the dynamic program. The proofs are based
on a case distinction according to the way an optimum edge open packing set
interacts with the twin set of the graph produced at the current node.

%
%
%
%
%
%
%

\begin{lemma}\label{dh_lemma_1}
	Let \(G=G_l\otimes G_r\). Then the following hold.
	\begin{enumerate}[label=\textup{(\roman*)}]
		\item \(\rho_1(G)=\max\{\rho_1(G_l)+\rho_3(G_r),\rho_3(G_l)+\rho_1(G_r),
		\rho_4(G_l)+\rho_5(G_r)+p_l+q_r,\rho_5(G_l)+\rho_4(G_r)+p_r+q_l\}\).
		
		\item \(\rho_2(G)=\max\{\rho_2(G_l)+\rho_3(G_r),
		\rho_3(G_l)+\rho_2(G_r)\}\).
		
		\item \(\rho_3(G)=\rho_3(G_l)+\rho_3(G_r)\).
		
		\item \(\rho_4(i,G)=\max\{\rho_4(i,G_l)+\rho_3(G_r),
		\rho_3(G_l)+\rho_4(i,G_r)\}\), and hence
		\(\rho_4(G)=\max_i\rho_4(i,G)\).
		
		\item \(\rho_5(j,G)=\max\{\rho_5(j,G_l)+\rho_3(G_r),
		\rho_3(G_l)+\rho_5(j,G_r)\}\), and hence
		\(\rho_5(G)=\max_j\rho_5(j,G)\).
	\end{enumerate}
\end{lemma}

\begin{proof}
	We use the fact that, in the true twin operation, \(TS(G)=TS(G_l)\cup TS(G_r)\) and every vertex of \(TS(G_l)\) is adjacent to every vertex of \(TS(G_r)\). This complete adjacency is the only interaction between the two sides relevant to the recurrences. In particular, whenever a selected edge uses the twin set on one side, any selected edge touching the twin set on the other side may create a forbidden common edge, unless all such selected edges belong to the same induced star.
	
	We first prove the recurrence for \(\rho_1(G)\). Let \(S\in P_1(G)\) be an optimum set. Since \(S\) contains an edge inside \(E(TS(G))\), there are three possibilities. If \(S\) contains an edge inside \(TS(G_l)\), then the right side must avoid \(TS(G_r)\); otherwise the complete adjacency between \(TS(G_l)\) and \(TS(G_r)\) creates a forbidden common edge. Thus the contribution is at most \(\rho_1(G_l)+\rho_3(G_r)\), and this value is attainable by combining an optimum set from \(P_1(G_l)\) with an optimum set from \(P_3(G_r)\). Symmetrically, if the selected twin edge lies inside \(TS(G_r)\), the contribution is exactly \(\rho_3(G_l)+\rho_1(G_r)\).
	
	It remains to consider the case where \(S\) contains a selected cross edge between \(TS(G_l)\) and \(TS(G_r)\). Then all selected cross edges must form one induced star. If the centre of this star lies in \(TS(G_l)\), then the selected leaves in \(TS(G_r)\) form an independent set counted by the \(P_4\)-condition on \(G_r\), and the selected external leaves adjacent to the centre on the left side are controlled by the \(P_5\)-condition on \(G_l\). Therefore, this case contributes \(\rho_5(G_l)+\rho_4(G_r)+q_l+p_r\). If the centre lies in \(TS(G_r)\), the symmetric contribution is \(\rho_4(G_l)+\rho_5(G_r)+p_l+q_r\). Hence, 
	\(\rho_1(G)=\max\{\rho_1(G_l)+\rho_3(G_r),\rho_3(G_l)+\rho_1(G_r),\rho_4(G_l)+\rho_5(G_r)+p_l+q_r,\rho_5(G_l)+\rho_4(G_r)+p_r+q_l\}\).
	
	For \(\rho_2(G)\), let \(S\in P_2(G)\) be optimum. Then \(S\) contains no selected edge inside \(TS(G)\), but it contains a selected edge with exactly one endpoint in \(TS(G)\). If this endpoint lies in \(TS(G_l)\), then the right side must avoid \(TS(G_r)\), and the contribution is \(\rho_2(G_l)+\rho_3(G_r)\). If the endpoint lies in \(TS(G_r)\), the symmetric contribution is \(\rho_3(G_l)+\rho_2(G_r)\). Both values are attainable by taking the corresponding optimum packings on the two sides. Thus \(\rho_2(G)=\max\{\rho_2(G_l)+\rho_3(G_r),\rho_3(G_l)+\rho_2(G_r)\}\).
	
	For \(\rho_3(G)\), every selected edge avoids \(TS(G)\). Hence the two sides are independent with respect to the packing, and the restriction of any optimum set to each side belongs to \(P_3\). Conversely, two optimum \(P_3\)-packings on the two sides can be combined because both avoid their twin sets. Therefore \(\rho_3(G)=\rho_3(G_l)+\rho_3(G_r)\).
	
	Now fix an index \(i\). In a set counted by \(P_4(i,G)\), there is an independent set \(A\subseteq TS(G)\) of size \(i\). Since \(TS(G_l)\) is completely adjacent to \(TS(G_r)\), the set \(A\) must lie entirely in one side. If \(A\subseteq TS(G_l)\), then the left side contributes \(\rho_4(i,G_l)\), while the right side must avoid \(TS(G_r)\), contributing \(\rho_3(G_r)\). The symmetric case gives \(\rho_3(G_l)+\rho_4(i,G_r)\). Hence \(\rho_4(i,G)=\max\{\rho_4(i,G_l)+\rho_3(G_r),\rho_3(G_l)+\rho_4(i,G_r)\}\), and  \(\rho_4(G)=\max_i\rho_4(i,G)\).
	
	Finally, fix an index \(j\). A set counted by \(P_5(j,G)\) is witnessed by a vertex \(x\in TS(G)\) and an independent set \(B_x\subseteq N_{G-TS(G)}(x)\) of size \(j\). If \(x\in TS(G_l)\), then the left side contributes \(\rho_5(j,G_l)\), while the right side must avoid \(TS(G_r)\), contributing \(\rho_3(G_r)\). The case \(x\in TS(G_r)\) is symmetric. Therefore \(\rho_5(j,G)=\max\{\rho_5(j,G_l)+\rho_3(G_r),\rho_3(G_l)+\rho_5(j,G_r)\}\), and taking the maximum over \(j\) gives \(\rho_5(G)=\max_j\rho_5(j,G)\).
\end{proof}

\begin{lemma}\label{dh_lemma_2}
	Let \(G=G_l\odot G_r\). Then the following hold.
	\begin{enumerate}[label=\textup{(\roman*)}]
		\item \(\rho_1(G)=\max\{\rho_1(G_l)+\rho_e^o(G_r),
		\rho_e^o(G_l)+\rho_1(G_r)\}\).
		
		\item \(\rho_2(G)=\max\{\rho_2(G_l)+\rho_2(G_r),
		\rho_2(G_l)+\rho_3(G_r),\rho_3(G_l)+\rho_2(G_r)\}\).
		
		\item \(\rho_3(G)=\rho_3(G_l)+\rho_3(G_r)\).
		
		\item \(\rho_4(i,G)=\max\{\rho_4(i,G_l)+\rho_3(G_r),
		\rho_3(G_l)+\rho_4(i,G_r),
		\rho_4(i_1,G_l)+\rho_4(i_2,G_r):i=i_1+i_2\}\), and hence
		\(\rho_4(G)=\max_i\rho_4(i,G)\).
		
		\item \(\rho_5(j,G)=\max\{\rho_5(j,G_l)+\rho_3(G_r),
		\rho_3(G_l)+\rho_5(j,G_r)\}\), and hence
		\(\rho_5(G)=\max_j\rho_5(j,G)\).
	\end{enumerate}
\end{lemma}

	\begin{proof}
		In the false twin operation, there are no edges between \(G_l\) and \(G_r\), and \(TS(G)=TS(G_l)\cup TS(G_r)\). Hence, an edge open packing set in \(G\) is obtained by combining edge open packing sets from the two sides, and no new conflict is created between a selected edge of \(G_l\) and a selected edge of \(G_r\). The only thing that has to be tracked is how the chosen sets interact with the new twin set \(TS(G)\).
		
		For \(\rho_1(G)\), a counted set contains a selected edge inside \(TS(G)\). Since there are no edges between \(G_l\) and \(G_r\), such an edge lies either inside \(TS(G_l)\) or inside \(TS(G_r)\). In the first case, the left side contributes \(\rho_1(G_l)\), while the right side may be any edge open packing set, contributing \(\rho_e^o(G_r)\). The second case is symmetric. Therefore,  \(\rho_1(G)=\max\{\rho_1(G_l)+\rho_e^o(G_r),\rho_e^o(G_l)+\rho_1(G_r)\}\).
		
		For \(\rho_2(G)\), a counted set contains no selected edge inside \(TS(G)\), but it contains at least one selected edge with exactly one endpoint in \(TS(G)\). This may happen on the left side only, on the right side only, or on both sides. These three cases give the values \(\rho_2(G_l)+\rho_3(G_r)\), \(\rho_3(G_l)+\rho_2(G_r)\), and \(\rho_2(G_l)+\rho_2(G_r)\), respectively. Since the two sides are anticomplete, all three combinations are feasible, and hence \(\rho_2(G)=\max\{\rho_2(G_l)+\rho_2(G_r),\rho_2(G_l)+\rho_3(G_r),\rho_3(G_l)+\rho_2(G_r)\}\).
		
		For \(\rho_3(G)\), a counted set avoids the whole twin set \(TS(G)\). Therefore its restrictions to \(G_l\) and \(G_r\) belong to \(P_3(G_l)\) and \(P_3(G_r)\), respectively. Conversely, any two such packings can be combined because there are no edges between the two sides. Thus \(\rho_3(G)=\rho_3(G_l)+\rho_3(G_r)\).
		
		Now fix an index \(i\). A set counted by \(P_4(i,G)\) is a member of \(P_3(G)\), and it is witnessed by an independent set \(A\subseteq TS(G)\) of size \(i\). Since there are no edges between \(TS(G_l)\) and \(TS(G_r)\), the set \(A\) may lie completely in \(TS(G_l)\), completely in \(TS(G_r)\), or split between the two sides. If \(A\subseteq TS(G_l)\), then the left side contributes \(\rho_4(i,G_l)\), while the right side must still avoid \(TS(G_r)\), contributing \(\rho_3(G_r)\). If \(A\subseteq TS(G_r)\), the symmetric contribution is \(\rho_3(G_l)+\rho_4(i,G_r)\). Finally, if \(A=A_l\cup A_r\), where \(A_l\subseteq TS(G_l)\), \(A_r\subseteq TS(G_r)\), \(|A_l|=i_1\), \(|A_r|=i_2\), and \(i=i_1+i_2\), then both sides must satisfy the corresponding \(P_4\)-conditions, giving \(\rho_4(i_1,G_l)+\rho_4(i_2,G_r)\). Therefore \(\rho_4(i,G)=\max\{\rho_4(i,G_l)+\rho_3(G_r),\rho_3(G_l)+\rho_4(i,G_r),\rho_4(i_1,G_l)+\rho_4(i_2,G_r):i=i_1+i_2\}\). Taking the maximum over \(i\) gives \(\rho_4(G)=\max_i\rho_4(i,G)\).
		
		Finally, fix an index \(j\). A set counted by \(P_5(j,G)\) is also a member of \(P_3(G)\). It is witnessed by a vertex \(x\in TS(G)\) and an independent set \(B_x\subseteq N_{G-TS(G)}(x)\) of size \(j\). If \(x\in TS(G_l)\), then, because there are no edges between \(G_l\) and \(G_r\), the set \(B_x\) lies entirely in \(G_l\). Thus the left side contributes \(\rho_5(j,G_l)\), while the right side must avoid \(TS(G_r)\), contributing \(\rho_3(G_r)\). The case \(x\in TS(G_r)\) is symmetric and gives \(\rho_3(G_l)+\rho_5(j,G_r)\). Hence \(\rho_5(j,G)=\max\{\rho_5(j,G_l)+\rho_3(G_r),\rho_3(G_l)+\rho_5(j,G_r)\}\), and taking the maximum over \(j\) gives \(\rho_5(G)=\max_j\rho_5(j,G)\).
	\end{proof}

\begin{lemma}\label{dh_lemma_3}
	Let \(G=G_l\oplus G_r\). Then the following hold.
	\begin{enumerate}[label=\textup{(\roman*)}]
		\item \(\rho_1(G)=\rho_1(G_l)+\rho_3(G_r)\).
		
		\item \(\rho_2(G)=\max\{\rho_2(G_l)+\rho_3(G_r),
		\rho_4(G_l)+\rho_5(G_r)+p_l+q_r,
		\rho_5(G_l)+\rho_4(G_r)+p_r+q_l\}\).
		
		\item \(\rho_3(G)=\rho_3(G_l)+\rho_e^o(G_r)\).
		
		\item \(\rho_4(i,G)=\rho_4(i,G_l)+\rho_3(G_r)\), and hence
		\(\rho_4(G)=\max_i\rho_4(i,G)\).
		
		\item \(\rho_5(j,G)=\max\{\rho_5(j,G_l)+\rho_3(G_r),
		\rho_5(j_1,G_l)+\rho_4(j_2,G_r):j=j_1+j_2\text{ and }j_2\neq 0\}\), and
		hence \(\rho_5(G)=\max_j\rho_5(j,G)\).
	\end{enumerate}
\end{lemma}

\begin{proof}
	We use the defining property of the attachment operation. In \(G=G_l\oplus G_r\), every vertex of \(TS(G_l)\) is adjacent to every vertex of \(TS(G_r)\), but the new twin set is \(TS(G)=TS(G_l)\). Thus only the left twin set remains active as the twin set of the combined graph.
	
	We first consider \(\rho_1(G)\). A set counted by \(P_1(G)\) contains a selected edge inside \(TS(G)\). Since \(TS(G)=TS(G_l)\), such an edge must lie inside \(TS(G_l)\). Once such an edge is selected, the right side must avoid \(TS(G_r)\); otherwise the complete adjacency between \(TS(G_l)\) and \(TS(G_r)\) would create a forbidden common edge. Hence the left side contributes \(\rho_1(G_l)\), and the right side contributes \(\rho_3(G_r)\). Conversely, an optimum set from \(P_1(G_l)\) together with an optimum set from \(P_3(G_r)\) gives a valid member of \(P_1(G)\). Therefore \(\rho_1(G)=\rho_1(G_l)+\rho_3(G_r)\).
	
	Next consider \(\rho_2(G)\). A set counted by \(P_2(G)\) contains no selected edge inside \(TS(G_l)\), but it contains at least one selected edge with exactly one endpoint in \(TS(G_l)\). There are two possible ways this can happen. The selected edge may lie completely in \(G_l\), in which case the left side contributes \(\rho_2(G_l)\), and the right side must avoid \(TS(G_r)\). This gives the term \(\rho_2(G_l)+\rho_3(G_r)\). The other possibility is that the selected edges with one endpoint in \(TS(G_l)\) are cross edges between \(TS(G_l)\) and \(TS(G_r)\). Since the selected cross edges must form one induced star, the same star-counting argument used for the true twin operation applies. If the star is represented by a \(P_4\)-condition on \(G_l\) and a \(P_5\)-condition on \(G_r\), we obtain \(\rho_4(G_l)+\rho_5(G_r)+p_l+q_r\). If the roles are reversed, we obtain \(\rho_5(G_l)+\rho_4(G_r)+p_r+q_l\). The avoidance conditions in \(P_3\), \(P_4\), and \(P_5\) ensure that these constructions introduce no forbidden common edge. Hence \(\rho_2(G)=\max\{\rho_2(G_l)+\rho_3(G_r),\rho_4(G_l)+\rho_5(G_r)+p_l+q_r,\rho_5(G_l)+\rho_4(G_r)+p_r+q_l\}\).
	
	For \(\rho_3(G)\), a counted set avoids \(TS(G)=TS(G_l)\). Hence the left part must belong to \(P_3(G_l)\). The right side is not part of the new twin set, and therefore it may contain any edge open packing set of \(G_r\). Conversely, any optimum set from \(P_3(G_l)\) can be combined with any maximum edge open packing set of \(G_r\), because no selected edge of the left part is incident with \(TS(G_l)\). Thus \(\rho_3(G)=\rho_3(G_l)+\rho_e^o(G_r)\).
	
	Now fix an index \(i\). A set counted by \(P_4(i,G)\) is witnessed by an independent set \(A\subseteq TS(G)=TS(G_l)\) of size \(i\). Since every vertex of \(A\) is adjacent to every vertex of \(TS(G_r)\), the right side must avoid \(TS(G_r)\). Hence the left side contributes \(\rho_4(i,G_l)\), and the right side contributes \(\rho_3(G_r)\). The reverse construction is obtained by combining optimum sets of these two types. Therefore \(\rho_4(i,G)=\rho_4(i,G_l)+\rho_3(G_r)\), and taking the maximum over \(i\) gives \(\rho_4(G)=\max_i\rho_4(i,G)\).
	
	Finally, fix an index \(j\). A set counted by \(P_5(j,G)\) is witnessed by a vertex \(x\in TS(G)=TS(G_l)\) and an independent set \(B_x\subseteq N_{G-TS(G)}(x)\) of size \(j\). Since \(x\) is adjacent to all vertices of \(TS(G_r)\), the set \(B_x\) may either avoid \(TS(G_r)\) or use some vertices of \(TS(G_r)\). If \(B_x\cap TS(G_r)=\emptyset\), then the left side contributes \(\rho_5(j,G_l)\), while the right side must avoid \(TS(G_r)\), giving \(\rho_5(j,G_l)+\rho_3(G_r)\). If \(B_x\cap TS(G_r)\neq\emptyset\), write \(B_x=B_x^1\cup B_x^2\), where \(B_x^1\subseteq V(G_l)\setminus TS(G_l)\), \(B_x^2\subseteq TS(G_r)\), \(|B_x^1|=j_1\), \(|B_x^2|=j_2\), \(j=j_1+j_2\), and \(j_2\neq 0\). Then the left side satisfies a \(P_5(j_1,G_l)\)-condition, while the right side satisfies a \(P_4(j_2,G_r)\)-condition. This gives the term \(\rho_5(j_1,G_l)+\rho_4(j_2,G_r)\). Conversely, any optimum sets of these corresponding types can be combined, since their avoidance conditions ensure that no selected edge is incident with the forbidden neighbourhoods. Therefore \(\rho_5(j,G)=\max\{\rho_5(j,G_l)+\rho_3(G_r),\rho_5(j_1,G_l)+\rho_4(j_2,G_r):j=j_1+j_2\text{ and }j_2\neq 0\}\), and taking the maximum over \(j\) gives \(\rho_5(G)=\max_j\rho_5(j,G)\).
\end{proof}

 \subsection{The algorithm}
 
 The algorithm processes the decomposition tree \(T\) bottom-up. For each node
 \(v\), it computes the values \(\rho_1(G_v)\), \(\rho_2(G_v)\),
 \(\rho_3(G_v)\), and \(\rho_e^o(G_v)\). It also stores the arrays
 \(\rho_4(i,G_v)\) and \(\rho_5(j,G_v)\) for all relevant indices
 \(0\le i,j\le |V(G_v)|\). Invalid entries are stored as \(-\infty\). The
 values \(\rho_4(G_v)\), \(\rho_5(G_v)\), and the largest indices attaining
 these maxima are maintained together with the arrays.
 
 At a leaf node \(v\), the graph \(G_v\) consists of one vertex and has no edge.
 Thus the empty set is the only edge open packing set. The table is initialized
 directly: \(\rho_3(G_v)=0\), \(\rho_e^o(G_v)=0\), and all families requiring a
 selected edge receive value \(-\infty\). The relevant zero-size avoidance
 entries of the auxiliary arrays are initialized consistently.
 
 At an internal node \(v\), let \(v_l\) and \(v_r\) be its children, and let
 \(G_l=G_{v_l}\) and \(G_r=G_{v_r}\). If \(v\) is labelled by \(\otimes\), then
 the table of \(G_v\) is computed using Lemma~\ref{dh_lemma_1}. If \(v\) is
 labelled by \(\odot\), then the table is computed using
 Lemma~\ref{dh_lemma_2}. If \(v\) is labelled by \(\oplus\), then the table is
 computed using Lemma~\ref{dh_lemma_3}. After all entries have been computed, we
 set \(\rho_e^o(G_v)=\max\{\rho_1(G_v),\rho_2(G_v),\rho_3(G_v)\}\), as justified
 by Observation~\ref{Obs: dh_graphs_1}.
 
 When the root \(r\) is processed, \(G_r=G\). Therefore the value stored as
 \(\rho_e^o(G_r)\) is the maximum edge open packing number of \(G\). By storing
 predecessor choices for the maximum values used in the recurrences, an actual
 maximum edge open packing set can also be recovered by backtracking from the
 root.
 
 \begin{theorem}\label{thm:dh-eop-correct}
 	The dynamic programming algorithm correctly computes \(\rho_e^o(G)\) for every
 	distance-hereditary graph \(G\).
 \end{theorem}
 
 \begin{proof}
 	We prove the claim by induction over the rooted decomposition tree. For a leaf
 	node, the claim is immediate because the corresponding graph has one vertex and
 	no edge. Now let \(v\) be an internal node with children \(v_l\) and \(v_r\),
 	and assume that the tables at \(v_l\) and \(v_r\) are correct. If \(v\) is
 	labelled by \(\otimes\), then Lemma~\ref{dh_lemma_1} gives exactly the optimum
 	values for all required families in \(G_v=G_l\otimes G_r\). If \(v\) is
 	labelled by \(\odot\), the same follows from Lemma~\ref{dh_lemma_2}. If \(v\)
 	is labelled by \(\oplus\), it follows from Lemma~\ref{dh_lemma_3}. Thus the
 	table at \(v\) is correct in all cases. By induction, the table at the root is
 	correct. Finally, Observation~\ref{Obs: dh_graphs_1} gives
 	\(\rho_e^o(G)=\max\{\rho_1(G),\rho_2(G),\rho_3(G)\}\). Hence the algorithm
 	correctly computes \(\rho_e^o(G)\).
 \end{proof}
 
 \subsection{Running time analysis}
 
 Let \(n=|V(G)|\). Since \(m\le n^2\), the decomposition tree \(T\) can be
 constructed in \(O(n^2)\) time by Lemma~\ref{lem:dh-decomposition-tree}. The
 tree has \(O(n)\) nodes. For a node \(v\), let \(n_v=|V(G_v)|\). The table at
 \(v\) stores a constant number of scalar values and two arrays of length at
 most \(n_v+1\), namely the arrays for \(\rho_4(i,G_v)\) and
 \(\rho_5(j,G_v)\). Hence, the table size at \(v\) is \(O(n_v)\).
 
 A leaf node is processed in constant time. At a true twin node, all scalar
 values are obtained from a constant number of expressions, and the arrays
 \(\rho_4(i,G_v)\) and \(\rho_5(j,G_v)\) are computed by scanning their indices.
 Thus a true twin node is processed in \(O(n_v)\) time. At a false twin node,
 the only convolution-type computation is the split \(i=i_1+i_2\) in
 Lemma~\ref{dh_lemma_2}. Computing all values of \(\rho_4(i,G_v)\) therefore
 takes \(O(n_v^2)\) time, while all other entries take \(O(n_v)\) time. Hence,  a
 false twin node is processed in \(O(n_v^2)\) time. At an attachment node, the
 only convolution-type computation is the split \(j=j_1+j_2\) in
 Lemma~\ref{dh_lemma_3}. Computing all values of \(\rho_5(j,G_v)\) therefore
 takes \(O(n_v^2)\) time, while all other entries take \(O(n_v)\) time. Hence, an
 attachment node is also processed in \(O(n_v^2)\) time.
 
 Therefore, every node \(v\) is processed in \(O(n_v^2)\) time. Since
 \(n_v\le n\) for every node and \(T\) has \(O(n)\) nodes, the total dynamic
 programming time is \(O(\sum_{v\in V(T)}n_v^2)\le O(n^3)\). Including the
 construction of the decomposition tree, the total running time remains
 \(O(n^3)\).
 
  If only the optimum value is required, child tables can
 be discarded after their parent has been processed. If a maximum edge open
 packing set is required, predecessor pointers may be stored without changing
 the polynomial bound.
 
 \begin{theorem}\label{thm:dist-here}
 	The \textsc{Maximum Edge Open Packing} problem can be solved in \(O(n^3)\) time
 	on distance-hereditary graphs.
 \end{theorem}
 
 \begin{proof}
 	The correctness follows from Theorem~\ref{thm:dh-eop-correct}. The running-time
 	analysis above shows that the decomposition tree can be constructed within
 	\(O(n^3)\) time and that the dynamic program over the decomposition tree takes
 	\(O(n^3)\) time. Hence the total running time is \(O(n^3)\).
 \end{proof}

\section{Biconvex bipartite graphs}\label{EOP_Sec_3}

\begin{theorem}\label{thm:biconvex}
    \maxeop~is polynomial time solvable in biconvex bipartite graphs. 
\end{theorem}

We use ``multi-chain ordering'' of biconvex bipartite graphs to obtain an algorithm for \maxeop. Let \( G = (V,E) \) be a biconvex bipartite graph with bipartition \( (X,Y) \). Biconvex bipartite graphs admit multi-chain ordering \cite{Diaz2021} and such an ordering can be constructed in polynomial time.  Let $L_0, L_1, \dots, L_p$ be the layers of a multi-chain ordering of $G$, where
\( L_{2i} \subseteq X \), \( L_{2i+1} \subseteq Y \) for $i=0, 1, 2, \ldots $.
This layered structure imposes a strong ordering on the edges of \( G \) and significantly restricts the possible conflicts
between edges in an edge open packing set. 
Exploiting these properties, we show that the selection of edges from each pair of consecutive layers
can be performed independently up to local consistency constraints, which leads to a polynomial-time dynamic programming algorithm for \maxeop.

We first look at some structural results before proceeding with the algorithm.

\begin{observation}\label{obs:parallel}
There cannot be two \emph{parallel solution edges} between two consecutive layers of a multi-chain ordering.
\end{observation}

\begin{proof}
Let \( G = (V,E) \) be a biconvex bipartite graph, let
\( L_0, L_1, \dots, L_p \) be the layers of a multi-chain ordering of \( G \),
and let \( D \subseteq E \) be an edge open packing set of \( G \).

Suppose, for the sake of contradiction, that \( D \) contains two parallel solution edges
\( e_1 = u_1v_1 \) and \( e_2 = u_2v_2 \), where
\( u_1, u_2 \in L_i \) and \( v_1, v_2 \in L_{i+1} \).
By the definition of a multi-chain ordering, the neighborhoods of vertices in \( L_i \)
into \( L_{i+1} \) are nested under set inclusion.
Without loss of generality, assume that
\( N(u_1) \cap L_{i+1} \subseteq N(u_2) \cap L_{i+1} \).

Since \( v_1 \in N(u_1) \cap L_{i+1} \), it follows that \( v_1 \in N(u_2) \cap L_{i+1} \),
and hence the edge \( u_2v_1 \in E \).
The edge \( u_2v_1 \) is distinct from both \( e_1 \) and \( e_2 \),
and it is incident to an endpoint of \( e_1 \) (namely \( v_1 \))
and to an endpoint of \( e_2 \) (namely \( u_2 \)). Thus, \( e_1 \) and \( e_2 \) have a common edge, contradicting the assumption that
\( D \) is an edge open packing set.
Therefore, no edge open packing set can contain two parallel solution edges between two consecutive layers. 
\end{proof}

Given a linear ordering $\sigma$ on a vertex set $U$ and two vertices $a,b \in U$
with $a \preceq_\sigma b$, we define the interval
$[a,b]_\sigma := \{\, v \in U \mid a \preceq_\sigma v \preceq_\sigma b \,\}$. 
If $b \prec_\sigma a$, then $[a,b]_\sigma := \emptyset$.

\begin{lemma}\label{lem:structural-lemma}
Let \( w \in L_i \) be a vertex such that the vertices of a solution induce a star
on a subset of vertices in \( L_i \cup L_{i+1} \), denoted by \( S_w \), with center vertex \( w \).
Let \( \sigma_{i+1,1} \) be the ordering of \( L_{i+1} \) decreasing with respect to \( L_i \),
and let \( \sigma_{i+1,2} \) be the ordering of \( L_{i+1} \) increasing with respect to \( L_{i+2} \).

Assume that \( S_w \cap L_{i+1} \neq \emptyset \).
Let \( v_f, v_{\ell} \) be the first and last vertices of \( L_{i+1} \) in the ordering
\( \sigma_{i+1,1} \) that belong to \( S_w \),
and let \( u_f, u_{\ell} \) be the first and last vertices of \( L_{i+1} \) in the ordering
\( \sigma_{i+1,2} \) that belong to \( S_w \). Then
\[
S_w \cap L_{i+1}
=
\{\, v \in L_{i+1} \mid
v \in [v_f,v_{\ell}]_{\sigma_{i+1,1}}
\;\cap\;
[u_f,u_{\ell}]_{\sigma_{i+1,2}}
\,\}.
\]
In particular, the leaves of the star \( S_w \) in layer \( L_{i+1} \) form a contiguous block
in both orderings \( \sigma_{i+1,1} \) and \( \sigma_{i+1,2} \).
\end{lemma}

\begin{proof}
Let \( w \in L_i \) be the center of a solution star \( S_w \), and suppose
\( S_w \cap L_{i+1} \neq \emptyset \).
By definition of a solution star, every vertex in \( S_w \cap L_{i+1} \)
is adjacent to \( w \), and no other edges incident to these vertices
belong to the solution.

We first consider the ordering \( \sigma_{i+1,1} \), which orders
\( L_{i+1} \) decreasingly with respect to \( L_i \).
Let \( v_f \) and \( v_{\ell} \) be the first and last vertices of
\( L_{i+1} \) under \( \sigma_{i+1,1} \) that belong to \( S_w \).

\medskip
\noindent\emph{Claim 1.}
If \( v \in L_{i+1} \) satisfies
\( v_f \prec_{\sigma_{i+1,1}} v \prec_{\sigma_{i+1,1}} v_{\ell} \),
then \( v \) is adjacent to \( w \).

\smallskip
\noindent
Indeed, since \( v_f \prec_{\sigma_{i+1,1}} v \),
by the definition of the multi-chain ordering we have
$N(v_f) \cap L_i \supseteq N(v) \cap L_i$. 
Since \( w \in N(v_f) \cap L_i \), it follows that
\( w \in N(v) \cap L_i \), and hence \( vw \in E(G) \).
Now consider the ordering \( \sigma_{i+1,2} \), which orders
\( L_{i+1} \) increasingly with respect to \( L_{i+2} \).
Let \( u_f \) and \( u_{\ell} \) be the first and last vertices of
\( L_{i+1} \) under \( \sigma_{i+1,2} \) that belong to \( S_w \).

\medskip
\noindent\emph{Claim 2.}
If \( v \in L_{i+1} \) satisfies
\( u_f \prec_{\sigma_{i+1,2}} v \prec_{\sigma_{i+1,2}} u_{\ell} \),
then every neighbor of \( v \) in \( L_{i+2} \) is also a neighbor of
\( u_{\ell} \) in \( L_{i+2} \).

\smallskip
\noindent
This follows directly from the definition of \( \sigma_{i+1,2} \),
since for vertices ordered increasingly with respect to \( L_{i+2} \),
we have $N(u_{\ell}) \cap L_{i+2} \supseteq N(v) \cap L_{i+2}$.

We now show the equality stated in the lemma. 
We first show the subset containment. 
Let \( v \in S_w \cap L_{i+1} \).
By definition of \( v_f, v_{\ell} \) and \( u_f, u_{\ell} \),
we must have
$v_f \preceq_{\sigma_{i+1,1}} v \preceq_{\sigma_{i+1,1}} v_{\ell}
\quad\text{and}\quad
u_f \preceq_{\sigma_{i+1,2}} v \preceq_{\sigma_{i+1,2}} u_{\ell}$. Hence \( v \) belongs to the stated set.

Let \( v \in L_{i+1} \) satisfy
$v_f \preceq_{\sigma_{i+1,1}} v \preceq_{\sigma_{i+1,1}} v_{\ell}
\quad\text{and}\quad
u_f \preceq_{\sigma_{i+1,2}} v \preceq_{\sigma_{i+1,2}} u_{\ell}$. 
By Claim~1, the vertex \( v \) is adjacent to \( w \).
Suppose, for contradiction, that \( v \notin S_w \). Then adding the edge \( wv \) to the solution would preserve the star
structure centered at \( w \) within \( L_i \cup L_{i+1} \).
Moreover, by Claim~2, any potential conflict created by including \( v \)
with a solution edge incident to \( L_{i+2} \) would also arise from
including \( u_{\ell} \), which already belongs to \( S_w \).
Thus, the inclusion of \( v \) would not create a new common-edge conflict,
contradicting the maximality and validity of the solution. Therefore, \( v \in S_w \), and hence 
$S_w \cap L_{i+1}
=
\{\, v \in L_{i+1} \mid
v \in [v_f,v_{\ell}]_{\sigma_{i+1,1}}
\;\cap\;
[u_f,u_{\ell}]_{\sigma_{i+1,2}}
\,\}$. 
In particular, the leaves of the star \( S_w \) in layer \( L_{i+1} \)
form a contiguous block in both orderings
\( \sigma_{i+1,1} \) and \( \sigma_{i+1,2} \). 
\end{proof}

\subsection{Solution types and overview of the algorithm}
We now describe how a solution to \maxeop~can intersect with a layer $L_i$ in a multi-chain ordering of a biconvex bipartite graph. 
The intersection of a solution with $L_i$ can essentially take the form of a \emph{star}, denoted $S_w$, 
with  center vertex $w$ and the remaining vertices forming the leaves. 
We distinguish two types of stars:


 

\noindent
\underline{\textbf{Type~1: Star centered in \(L_i\).}}  
The center vertex \(w\) belongs to \(L_i\), and the leaves of the star belong to
one or both adjacent layers \(L_{i-1}\) and \(L_{i+1}\). That is,
$S_w = \{w\} \cup S^-_w \cup S^+_w,
\quad
S^-_w \subseteq L_{i-1},
\quad
S^+_w \subseteq L_{i+1}$. 
For each non-empty leaf set \(S^-_w\) (resp.\ \(S^+_w\)) there exist endpoints
\(a^-, b^-\) (resp.\ \(a^+, b^+\)) and \(c^-, d^-\) (resp.\ \(c^+, d^+\)) such that:
\begin{itemize}
    \item \(a^-\) and \(b^-\) are the first and last vertices of \(S^-_w\)
    in the ordering \(\sigma_{i-1,1}\);
    \item \(c^-\) and \(d^-\) are the first and last vertices of \(S^-_w\)
    in the ordering \(\sigma_{i-1,2}\);
    \item \(a^+\) and \(b^+\) are the first and last vertices of \(S^+_w\)
    in the ordering \(\sigma_{i+1,1}\);
    \item \(c^+\) and \(d^+\) are the first and last vertices of \(S^+_w\)
    in the ordering \(\sigma_{i+1,2}\).
\end{itemize}


Then every vertex \(v \in S^-_w\) satisfies
\[
a^- \preceq_{\sigma_{i-1,1}} v \preceq_{\sigma_{i-1,1}} b^-
\quad \text{and} \quad
c^- \preceq_{\sigma_{i-1,2}} v \preceq_{\sigma_{i-1,2}} d^-,
\]
that is, each vertex of \(S^-_w\) appears between \(a^-\) and \(b^-\) in the ordering \(\sigma_{i-1,1}\), and simultaneously between \(c^-\) and \(d^-\) in the ordering \(\sigma_{i-1,2}\). Consequently, the vertices of \(S^-_w\) form contiguous blocks in both orderings. Similarly, every vertex \(v \in S^+_w\) satisfies
\[
a^+ \preceq_{\sigma_{i+1,1}} v \preceq_{\sigma_{i+1,1}} b^+
\quad \text{and} \quad
c^+ \preceq_{\sigma_{i+1,2}} v \preceq_{\sigma_{i+1,2}} d^+,
\]
implying that the vertices of \(S^+_w\) also appear consecutively in both orderings \(\sigma_{i+1,1}\) and \(\sigma_{i+1,2}\). In particular, all leaves in \(S^-_w\) lie in \(L_{i-1}\) and are bounded by
endpoints in the ordering of \(L_{i-1}\), and all leaves in \(S^+_w\) lie in
\(L_{i+1}\) and are bounded by endpoints in the ordering of \(L_{i+1}\). 

\vspace{2mm}

\noindent
\underline{\textbf{Type~2: Star centered outside \(L_i\).}}  
The center vertex \(w\) belongs to an adjacent layer \(L_{i-1}\) or \(L_{i+1}\).
The star \(S_w\) may have leaves in several layers, but for the purpose of
the DP state at layer \(L_i\) we only record the leaves that belong to \(L_i\).

Formally, if \(S_w\) is a solution star with center \(w \in L_{i-1} \cup L_{i+1}\),
we denote by
$S_w^0 := S_w \cap L_i$ 
the set of leaves of \(S_w\) that lie in \(L_i\).
If \(S_w^0 \neq \emptyset\), then there exist endpoints \(a^0,b^0\) in
\(\sigma_{i,1}\) and \(c^0,d^0\) in \(\sigma_{i,2}\) such that every vertex
\(v \in S_w^0\) satisfies 
$a^0 \preceq_{\sigma_{i,1}} v \preceq_{\sigma_{i,1}} b^0
\text{ and }
c^0 \preceq_{\sigma_{i,2}} v \preceq_{\sigma_{i,2}} d^0$.

\medskip
\noindent\textbf{Overview of the algorithm.}
We solve \maxeop~by dynamic programming along a multi-chain ordering of a
biconvex bipartite graph.
Let us consider $L_0, L_1, \ldots, L_p$ be the layers of such an ordering.
The algorithm processes the layers sequentially, and at each step, considers a
window of three consecutive layers $L_{i-1}, L_i,$ and $L_{i+1}$.

A crucial structural property is that any edge open packing intersects a layer
$L_i$ only through stars of one of the two types described above.
Specifically, the solution either contains a \emph{Type~1} star centered at a
vertex of $L_i$ with its leaves in the adjacent layers $L_{i-1}$ and/or $L_{i+1}$,
or it contains a \emph{Type~2} star whose center lies in $L_{i-1}$ or $L_{i+1}$
and whose leaves lie in $L_i$.
By Observation~\ref{obs:parallel}, at most one star can use edges between any
two consecutive layers, and by Lemma~\ref{lem:structural-lemma}, the leaves of each
star in a given layer lie between two endpoints in each of the two relevant
orderings of that layer.

For each triple of consecutive layers $L_{i-1}, L_i,$ and $L_{i+1}$, we enumerate
all feasible local configurations describing how a solution may intersect $L_i$.
Each configuration specifies (i) the type of star involved (Type~1 or Type~2),
(ii) the location of the center vertex, and (iii) the first and last leaf vertices
in the relevant layer orderings. These endpoints determine the exact set of
vertices in the intersection of the two intervals induced by the orderings,
which in turn uniquely determines the set of edges selected between
$L_{i-1}$ and $L_i$, and between $L_i$ and $L_{i+1}$. The dynamic programming table stores, for each layer $L_i$, all configurations
that can be extended to a valid solution on the subgraph induced by layers
$L_0, \ldots, L_i$.
A transition between configurations at layers $L_{i-1}$ and $L_i$ is allowed if
the corresponding edge selections are compatible, that is, no two edges chosen
in different layers have a common edge.
Since conflicts are confined to edges whose endpoints lie in adjacent layers,
this compatibility check is local and can be performed in polynomial time. By propagating all valid configurations from $L_0$ to $L_p$ and keeping track of
the number of selected edges, the algorithm computes a maximum edge open packing
of $G$.
As the number of configurations per layer is polynomially bounded, the overall
algorithm runs in polynomial time, completing the proof of
Theorem~\ref{thm:biconvex}.

\subsection{Dynamic programming table and states.}

The dynamic programming algorithm processes the layers
$L_0, L_1, \ldots, L_p$ in order.
For each layer $L_i$, we consider all possible ways in which a solution to \maxeop~can intersect the three-layer window
$L_{i-1} \cup L_i \cup L_{i+1}$.

Formally, the dynamic programming table is indexed by triples $(i, \Gamma, x)$,
where:
\begin{itemize}
    \item $i \in \{0,1,\ldots,p\}$ denotes the current layer;
    \item $\Gamma$ encodes a local configuration describing how the solution
    intersects layer $L_i$; and
    \item $x$ denotes the total number of edges selected in the partial solution
    induced by layers $L_0, \ldots, L_{i+1}$.
\end{itemize}

The configuration $\Gamma$ specifies:
(i) the type of star intersecting $L_i$ (Type~1 or Type~2),
(ii) the location of the center vertex (in $L_{i-1}$, $L_i$, or $L_{i+1}$, as
allowed by the star type), and
(iii) the first and last leaf vertices in the appropriate layer orderings.

In particular, for each non-empty leaf set in an adjacent layer, the configuration
stores two pairs of endpoints: one pair in the ordering $\sigma_{\cdot,1}$ and
one pair in the ordering $\sigma_{\cdot,2}$.  
By Lemma~\ref{lem:structural-lemma}, these endpoints determine exactly which
vertices of the layer belong to the star (as the intersection of the two induced
intervals), and hence uniquely determine the set of edges selected between
$L_{i-1}$ and $L_i$, as well as between $L_i$ and $L_{i+1}$.

A table entry $(i,\Gamma,x)$ is feasible if there exists an edge open packing on
the subgraph induced by layers $L_0,\ldots,L_{i+1}$ that realizes the
configuration $\Gamma$ at layer $L_i$ and contains exactly $x$ edges.

\medskip
\noindent\textbf{Formal description of a guess $\Gamma$.}
Fix a layer $L_i$.
A guess $\Gamma$ describes how a solution to \maxeop~ 
intersects the
three-layer window $L_{i-1} \cup L_i \cup L_{i+1}$.
It consists of the following components.

\begin{enumerate}
    \item A type parameter $\tau_i \in \{0,1,2,3\}$, where
    \begin{itemize}
        \item $\tau_i = 0$: no star intersects $L_i$;
        \item $\tau_i = 1$: only a \emph{Type~1} star intersects $L_i$;
        \item $\tau_i = 2$: only a \emph{Type~2} star intersects $L_i$;
        \item $\tau_i = 3$: both a \emph{Type~1} and a \emph{Type~2} star intersect $L_i$.
    \end{itemize}

    \item If $\tau_i \in \{1,3\}$ (Type~1 present), then $\Gamma$ specifies:
    \begin{itemize}
        \item a center vertex $w_1 \in L_i$; and
        \item for each adjacent layer $L_{i+\delta}$, $\delta \in \{-1,+1\}$,
        the first and last leaves of the star in that layer, according to both
        orderings of $L_{i+\delta}$:
        $a_{1,\delta}, b_{1,\delta} \in L_{i+\delta} \text{ (w.r.t.\ }\sigma_{i+\delta,1}\text{) }
        \text{and }
        c_{1,\delta}, d_{1,\delta} \in L_{i+\delta} \quad\text{(w.r.t.\ }\sigma_{i+\delta,2}\text{)}$. 
        
        The leaves of the star in $L_{i+\delta}$ are exactly those vertices
        lying between these endpoints in both orderings, i.e.,
         $S_{w_1} \cap L_{i+\delta}
        = [a_{1,\delta}, b_{1,\delta}]_{\sigma_{i+\delta,1}}
          \cap [c_{1,\delta}, d_{1,\delta}]_{\sigma_{i+\delta,2}}$. If the star has no leaves in $L_{i+\delta}$, then the corresponding
        endpoints are marked as empty.
    \end{itemize}

    \item If $\tau_i \in \{2,3\}$ (Type~2 present), then $\Gamma$ specifies:
    \begin{itemize}
        \item for each $\delta \in \{-1,+1\}$, an optional center vertex
        $w_{2,\delta} \in L_{i+\delta}$ (at most one per layer); 
        \item for each chosen center $w_{2,\delta}$, the first and last leaves of
        the star in layer $L_i$ according to both orderings of $L_i$:
            $a_{2,\delta}, b_{2,\delta} \in L_i \text{ (w.r.t.\ }\sigma_{i,1}\text{)}$
            $\text{and } 
        c_{2,\delta}, d_{2,\delta} \in L_i \quad\text{(w.r.t.\ }\sigma_{i,2}\text{)}$.

        The leaves of the star in $L_i$ are exactly those vertices lying between
        these endpoints in both orderings, i.e.,
        $S_{w_{2,\delta}} \cap L_i
        = [a_{2,\delta}, b_{2,\delta}]_{\sigma_{i,1}}
          \cap [c_{2,\delta}, d_{2,\delta}]_{\sigma_{i,2}}$. 
        If the center contributes no leaves in $L_i$, then the corresponding
        endpoints are marked as empty.
    \end{itemize}
\end{enumerate}

By Observation~\ref{obs:parallel}, for each $\delta \in \{-1,+1\}$, at most one  
Type~1 star and at most one Type~2 star may use edges between $L_i$ and
$L_{i+\delta}$.
By Lemma~\ref{lem:structural-lemma}, the leaves of each star in a given layer
are exactly the intersection of the intervals determined by the endpoints in
both orderings. 

\begin{lemma}[Invalid guess $\Gamma$]
\label{lem:invalid-tuples-independent}
Let $L_i$ be a layer of a biconvex bipartite graph with a multi-chain ordering, and
let $\Gamma$ be a guess describing the stars intersecting $L_i$.
Assume that the subgraph induced by each layer $L_i$ is independent.
The guess $\Gamma$ is \emph{invalid} if any of the following conditions hold:

\begin{enumerate}
    \item For some pair of endpoints $(a,b)$ or $(c,d)$ specified by $\Gamma$
    for a leaf set, the ordering is violated in the corresponding layer, i.e.,
    $a \succ_{\sigma_{j,1}} b$ or $c \succ_{\sigma_{j,2}} d$ for 
    layer $L_j$.

    \item A center vertex is chosen outside its allowed layer(s), namely,
    $w_1 \notin L_i$ for a Type~1 star, or
    $w_{2,\delta} \notin L_{i+\delta}$ for a Type~2 star.

    \item The edges induced by $\Gamma$ contain two edges that are the end-edges of
    a $P_4$ in $G$. 

    \item A star specified by $\Gamma$ has no leaves in any layer, i.e.,
    all the associated leaf sets are empty. 
\end{enumerate}

Any guess $\Gamma$ that satisfies none of the above conditions is called
\emph{valid}.
\end{lemma}

\noindent\textbf{DP Table Entries.}
We define
\[
\DP[i, \Gamma, x] =
\begin{cases}
\text{true,} &
\text{if there exists a partial solution of size $x$} \\
& \text{on layers $L_0,\ldots,L_{i+1}$ realizing $\Gamma$ at $L_i$,} \\
\text{false,} & \text{otherwise.}
\end{cases}
\]
\medskip
\noindent\textbf{Realizing a Guess $\Gamma$.}
A partial solution $T \subseteq E$ on layers $L_0,\ldots,L_{i+1}$ is said to
\emph{realize} a guess $\Gamma$ at layer $L_i$ if the following conditions hold:

\begin{enumerate}
    \item For each Type~1 star specified in $\Gamma$ with center $w_1 \in L_i$,
    the solution $T$ contains exactly the edges
        $w_1 v \quad \text{for all } v \in S_{w_1} \cap L_{i+\delta}$, 
    where $S_{w_1} \cap L_{i+\delta}
    = [a_{1,\delta},b_{1,\delta}]_{\sigma_{i+\delta,1}}
      \;\cap\;
      [c_{1,\delta},d_{1,\delta}]_{\sigma_{i+\delta,2}}
    \quad (\delta\in\{-1,+1\})$, and $T$ contains no other edges incident to $w_1$.

    \item For each Type~2 star specified in $\Gamma$ with center 
    $w_{2,\delta} \in L_{i+\delta}$, the solution $T$ contains exactly the edges $w_{2,\delta} v \quad \text{for all } v \in S_{w_{2,\delta}} \cap L_i$, 
    where $S_{w_{2,\delta}} \cap L_i
    = [a_{2,\delta},b_{2,\delta}]_{\sigma_{i,1}}
      \;\cap\;
      [c_{2,\delta},d_{2,\delta}]_{\sigma_{i,2}}
    \quad (\delta\in\{-1,+1\})$, 
    and $T$ contains no other edges incident to $w_{2,\delta}$.

    \item The set of edges $T$ restricted to layers $L_0,\ldots,L_{i+1}$ is an
    edge open packing.
    Equivalently, since $G$ is bipartite, no two edges in $T$ are the end-edges
    of a $P_4$.

    \item The stars induced by $T$ on layer $L_i$ correspond exactly to those
    encoded by $\Gamma$; in particular, the centers and the first and last
    vertices of each leaf set match those specified in $\Gamma$.
\end{enumerate}

Intuitively, a partial solution realizes $\Gamma$ if, when restricted to the layers
$L_0,\ldots,L_{i+1}$, the selected edges form precisely the stars described by
$\Gamma$ at layer $L_i$ and satisfy the edge open packing constraints.







\medskip
\noindent\textbf{Computation of a DP    entry.}
To fill the DP table, we iterate over all entries
$(i-1,\Gamma',x')$ and $(i,\Gamma)$.
Let $\Delta(\Gamma)$ denote the number of edges described by $\Gamma$ that are
incident to $L_{i+1}$ (and hence are counted for the first time when processing
layer $L_i$). A transition from $(i-1,\Gamma',x')$ to $(i,\Gamma,x)$, where
$x = x' + \Delta(\Gamma)$, is allowed only if $\Gamma'$ and $\Gamma$ are compatible. Here compatibility means:

\begin{itemize}
    \item The stars described by $\Gamma'$ and $\Gamma$ induce the same structure
    on layer $L_i$ (same centers and same leaf blocks in $L_i$);

    \item No edge selected by $\Gamma'$ shares a common edge with any edge
    selected by $\Gamma$.
    Equivalently, in the bipartite graph, no selected edges form the two
    end-edges of a $P_4$. Such a condition can be checked in polynomial time given the guesses $\Gamma'$ and $\Gamma$. 
\end{itemize}

If compatible, we set: 
$\DP[i,\Gamma, x' + \Delta(\Gamma)] \;:=\; \text{true}$. 

\medskip
\noindent\textbf{Answer.}
After processing all layers, we obtain the optimum value as
$$\max\{x \mid \DP[p,\Gamma,x] = \text{true} \text{ for some terminal guess } \Gamma\}$$
equal to the size of a maximum edge open packing in $G$. A guess $\Gamma$ at layer $L_p$ is terminal if it is valid and it does not
require any edges incident to a non-existent layer $L_{p+1}$.

\begin{lemma}
\label{lem:dp-correctness}
The recurrence described above is correct. 


\end{lemma}

We now prove that the values computed during the transitions are correct. Towards this, we will show the following equivalent lemma. 


	

\begin{lemma}
	Let $G$ be a biconvex bipartite graph with a multi-chain ordering
	$L_0, L_1, \dots, L_p$.
	For each layer $L_i$, let $\Gamma$ be a valid guess of the local solution structure
	and let $x \in \mathbb{N}$.
	Then
	
	\[
	\DP[i, \Gamma, x] =
	\begin{cases}
		\text{true,} &
		\text{if there exists a partial solution of size $x$} \\
		& \text{on layers $L_0,\ldots,L_{i+1}$ realizing $\Gamma$ at $L_i$,} \\
		\text{false,} & \text{otherwise.}
	\end{cases}
	\]
	
	Moreover, after processing all layers, the maximum weight $x$ for which
	$\DP[p, \Gamma, x] = \text{true}$ for some terminal guess $\Gamma$
	equals the size of a maximum edge open packing in $G$.
\end{lemma}
\begin{proof}
	We prove the lemma by induction on the layer index $i$.
	
	\medskip
	\noindent\textbf{Induction hypothesis.}
	For every $i \in \{0,\dots,p\}$, every valid guess $\Gamma$ at layer $L_i$ and
	every $x \in \mathbb{N}$, the table entry $\DP[i,\Gamma,x]$ is true
	iff there exists a partial solution $T$ of size $x$ on layers
	$L_0,\dots,L_{i+1}$ that realizes $\Gamma$ at $L_i$.
	
	\medskip
	\noindent\textbf{Base case ($i=0$).}
	A partial solution on layers $L_0$ and $L_1$ consists solely of edges between
	$L_0$ and $L_1$.
	For any valid guess $\Gamma$ at layer $L_0$, the set of edges described by $\Gamma$
	is exactly the set of edges that can be chosen in the window $L_0 \cup L_1$.
	Hence, $\DP[0,\Gamma,x]$ is true if and only if the number of edges specified by
	$\Gamma$ equals $x$ and these edges form a valid edge open packing.
	This is precisely the definition of realizing $\Gamma$ at $L_0$.
	Therefore the base case holds.
	
	\medskip
	\noindent\textbf{Induction step.}
	Assume the hypothesis holds for all layers up to $i-1$.
	We prove it for layer $i$.
	
	Let $\Gamma$ be a valid guess at layer $L_i$ and let $x \in \mathbb{N}$.
	We must show that $\DP[i,\Gamma,x]$ is true
	iff there exists a partial solution of size $x$ on layers $L_0,\dots,L_{i+1}$
	realizing $\Gamma$ at $L_i$.
	
	\medskip
	\noindent\textbf{(If direction.)}
	Assume that there exists a partial solution $T$ of size $x$ on layers
	$L_0,\dots,L_{i+1}$ realizing $\Gamma$ at $L_i$.
	
	Let $T_i$ be the restriction of $T$ to layers $L_0,\dots,L_i$.
	Define $\Gamma'$ to be the guess at layer $L_{i-1}$ induced by $T_i$.
	Such a guess exists because $T_i$ uniquely determines the stars that intersect
	layer $L_{i-1}$, and by Lemma~\ref{lem:structural-lemma} these stars correspond to
	intersections of intervals in the two orderings.
	
	Let $x'$ be the number of edges of $T_i$ that are counted in the DP table up to
	layer $i$, i.e., edges with an endpoint in $L_i$.
	
	Since $T_i$ is a valid partial solution on layers $L_0,\dots,L_i$ realizing
	$\Gamma'$ at $L_{i-1}$, by the induction hypothesis we have
	$\DP[i-1,\Gamma',x'] = \text{true}$.
	
	Moreover, $T$ contains exactly the edges described by $\Gamma$ between
	$L_i$ and $L_{i+1}$, and these edges are not counted in $x'$.
	Thus
	\[
	x = x' + \Delta(\Gamma).
	\]
	
	Finally, since $T$ is an edge open packing, no edge selected in
	$T_i$ shares a common edge with any edge selected between $L_i$ and $L_{i+1}$.
	Therefore the guesses $\Gamma'$ and $\Gamma$ are compatible, and the DP transition
	from $(i-1,\Gamma',x')$ to $(i,\Gamma,x)$ is allowed.
	Hence the DP sets $\DP[i,\Gamma,x] = \text{true}$.
	
	\medskip
	\noindent\textbf{(Only-if direction.)}
	Assume $\DP[i,\Gamma,x] = \text{true}$.
	Then by definition of the DP table, there exists a previous entry
	$(i-1,\Gamma',x')$ such that:
	\begin{enumerate}
		\item $\DP[i-1,\Gamma',x'] = \text{true}$,
		\item $x = x' + \Delta(\Gamma)$, and
		\item $\Gamma'$ and $\Gamma$ are compatible.
	\end{enumerate}
	
	By the induction hypothesis, since $\DP[i-1,\Gamma',x']$ is true,
	there exists a partial solution $T'$ of size $x'$ on layers
	$L_0,\dots,L_i$ realizing $\Gamma'$ at $L_{i-1}$.
	
	Compatibility of $\Gamma'$ and $\Gamma$ implies that the edges described by
	$\Gamma$ between $L_i$ and $L_{i+1}$ do not share a common edge with any edge of
	$T'$.
	Therefore, the set
	\[
	T = T' \cup E(\Gamma)
	\]
	(where $E(\Gamma)$ is the set of edges specified by $\Gamma$)
	is a valid edge open packing on layers $L_0,\dots,L_{i+1}$.
	
	Since $|E(\Gamma)| = \Delta(\Gamma)$, the size of $T$ is
	\[
	|T| = |T'| + \Delta(\Gamma) = x' + \Delta(\Gamma) = x.
	\]
	
	By construction, $T$ realizes $\Gamma$ at layer $L_i$, because the stars and
	leaf sets specified by $\Gamma$ are exactly the edges added between
	$L_i$ and $L_{i+1}$.
	Hence a partial solution of size $x$ exists, proving the only-if direction.
	
	\medskip
	\noindent\textbf{Conclusion.}
	By induction, the DP table satisfies the stated invariant for all layers $i$.
	
	\medskip
	\noindent\textbf{Optimality.}
	A guess $\Gamma$ at layer $L_p$ is terminal if it is valid and does not require
	edges incident to a non-existent layer $L_{p+1}$.
	For any terminal guess $\Gamma$, a partial solution realizing $\Gamma$ on
	$L_0,\dots,L_{p+1}$ is in fact a complete solution on all layers of $G$.
	Therefore, the maximum value $x$ such that $\DP[p,\Gamma,x]$ is true over all
	terminal guesses $\Gamma$ equals the size of a maximum edge open packing in $G$.
	
	This completes the proof. 
\end{proof}

The algorithm first computes a multi-chain ordering of the biconvex bipartite graph
and then processes the layers sequentially using dynamic programming.
Each DP state encodes a valid local configuration of stars intersecting a layer,
and transitions correspond to extending a partial solution to the next layer while
preserving the edge open packing property.

By Lemma~\ref{lem:dp-correctness}, the DP table correctly captures all feasible
partial solutions, and the maximum value stored in the final layer equals the
size of a maximum edge open packing.
Since the number of valid guesses per layer is polynomial and each
transition can be verified in polynomial time, the entire algorithm runs in
polynomial time, proving Theorem~\ref{thm:biconvex}. 

\subsection{Running Time Analysis}

Let $n$ denote the number of vertices in the graph.  
For each layer $L_i$, a guess $\Gamma$ encodes:

\begin{itemize}
	\item Type~1 stars (possibly in both $L_{i-1}$ and $L_{i+1}$) with possibly different centers.  
	Accounting for all centers and leaf blocks in both orderings gives $O(n^{10})$ possibilities.  
	
	\item Type~2 stars (possibly two centers, one in $L_{i-1}$ and one in $L_{i+1}$),  
	with leaves in $L_i$ determined by endpoints in both orderings, also giving $O(n^{10})$ possibilities.  
	
	\item Both Type~1 and Type~2 stars can coexist, leading to at most $O(n^{20})$ guesses $\Gamma$ per layer.  
	
	\item The third DP parameter $x$ counts the number of edges selected in the partial solution.  
	Since the total number of edges is at most $n$, this adds a multiplicative factor of $O(n)$.  
\end{itemize}

\noindent For each layer, we check transitions between all compatible guesses in the previous layer.  
This requires $O(n^{20} \cdot n^{20}) = O(n^{40})$ operations per layer for the guess compatibility.  

\noindent Including the $x$ parameter, the total work per layer is $O(n \cdot n^{40}) = O(n^{41})$.  
With $O(n)$ layers, the overall running time of the dynamic programming algorithm is
\[
O(n) \cdot O(n^{41}) = O(n^{42}).
\]

\section{FPT algorithm on chordal graphs}\label{EOP_Sec_4}

\begin{theorem}\label{thm:chordal}
    \maxeop~can be solved in FPT time when parameterized by the clique number of a chordal graph. 
\end{theorem}

\begin{observation}\label{obs:bag-structure}
Let $G$ be a chordal graph and let $(T,\mathcal{B})$ be a nice-tree decomposition of $G$.
For any bag $B \in \mathcal{B}$, at most one edge of $G[B]$ belongs to an edge open packing of $G$. 
\end{observation}
\begin{proof}
	Since $G$ is chordal, every bag $B$ of the tree decomposition induces a clique.
	Suppose, for contradiction, that there exist two distinct edges
	$e_1 = uv$ and $e_2 = xy$ in $G[B]$ that are both selected in an edge open packing.
	
	If $e_1$ and $e_2$ are vertex-disjoint, then because $B$ is a clique, the edge $ux$
	belongs to $G[B]$. This edge is incident to an endpoint of $e_1$ and an endpoint of
	$e_2$, implying that $e_1$ and $e_2$ have a common edge, which contradicts the
	definition of an edge open packing.
	
	If $e_1$ and $e_2$ share exactly one endpoint, say $u = x$, then again since $B$ is a
	clique. Then  
	$vy\in E(G[B])$ 
	and is incident to an endpoint of each of  
	$e_1$ and $e_2$, a contradiction.
	
	Thus, no two distinct edges of $G[B]$ can both be selected in a feasible edge open
	packing. Hence, at most one edge of $G[B]$ can belong to a solution.
\end{proof}

\noindent
\textbf{Overview of the algorithm.}
We design a dynamic programming algorithm over a nice tree decomposition
$(T,\mathcal{B})$ of the input chordal graph $G$.
The algorithm processes the decomposition bottom-up, and for each node $t$,
maintains information about how a partial edge open packing intersects the bag
$B_t$ and the subgraph induced by the subtree rooted at $t$.
The key observation is that, due to the structure of chordal graphs, at most one
edge can be selected inside any bag, due to Observation \ref{obs:bag-structure}, which allows us to describe all feasible
partial solutions using a bounded amount of information per bag.
By carefully propagating this information through the introduce, forget, and
join nodes of the decomposition, we compute a maximum edge open packing in
fixed-parameter tractable time with respect to the clique number (equivalently,
the treewidth).

\medskip
\noindent
\textbf{Dynamic Programming States.}
Let $(T,\mathcal{B})$ be a nice tree decomposition of $G$. For a node $t \in V(T)$, let $B_t$ be its bag and let $T_t$ denote the subgraph of $G$ induced by all vertices appearing in the subtree rooted at $t$. We define a dynamic programming table $\DP[t,g,h,p]$, where the parameters are:

\begin{itemize}
    \item $g: B_t \to \{0,1\}$ records adjacency to solution edges outside the bag. 
    For $v \in B_t$, $g(v)=1$ if there exists $w \in V(T_t)\setminus B_t$ such that $vw$ is selected in the partial solution; otherwise $g(v)=0$.

    \item The set $h: B_t \to \{\emptyset, \text{center}, \text{leaf}\}$ encodes the type of each vertex in the bag:
at most one vertex can be a center, and at most one vertex can be a leaf. 

    \item $p \in \mathbb{N}$ is the number of edges in the partial solution in $T_t$.
\end{itemize}

An entry $\DP[t,g,h,p]$ is set to \texttt{true} if there exists a partial edge open packing of size $p$ in $T_t$ consistent with $g$ and $h$.



\medskip
\noindent\textbf{Leaf Node.}
If $t$ is a leaf node, then $B_t=\emptyset$ and we set 
$\DP[t,\emptyset,\emptyset,0] := \text{true}$. 
All other entries are false.

\medskip
\noindent\textbf{Introduce Node.}
Let $t$ be an introduce node with child $t'$ and
$B_t = B_{t'} \cup \{v\}$.
For each entry $\DP[t',g',h',p]=\text{true}$, we consider the following cases.

\begin{itemize}
    \item \textbf{Vertex $v$ is not incident to any selected edge.} We consider
    $g(v) = 0, h(v)= \emptyset, \text{ and for all } u \in B_{t'}, h(u) = h'(u), g(u) = g'(u)$. 
    Then set    $\DP[t, g, h, p] := \text{true}$.

\item \textbf{Vertex $v$ is incident to a selected edge $vw$ with $w \in B_{t'}$. } This is allowed only if one of the following holds:
\begin{enumerate}
    \item $h'(w) = \emptyset$ and $g'(w) = 0$ (so adding $vw$ does not create an induced $P_4$).

For this case, we consider 
$h(v) = \text{center}$ (resp. $h(v) = \text{leaf}$), if $h(w) = \text{leaf}$ (resp. $h(w) = \text{center}$), $g(u) = 1 \ \forall u \in B_t$.

    \item $h'(w) = \text{center}$. 

We consider $h(v)=\text{leaf}$, and $g(u) = 1 \ \forall u \in B_t$. 

\end{enumerate}

In both the cases, we set
$\DP[t, g, h, p+1] := \text{true}$. 
Notice that if $h'(w) = \text{leaf}$, this indicates that $w$ is already an endpoint of a solution edge $ww'$ in $T_{t'}$, where $w'$ is the center. Since $w' \notin B_t$ and $w$ is a leaf, $w$ cannot be incident to any additional edge in future. Therefore, this case can be safely ignored.


\end{itemize}

\noindent\textbf{Forget Node.}
Let $t$ be a forget node with child $t'$ and
$B_t = B_{t'} \setminus \{v\}$. For each entry $\DP[t', g', h', p'] = \text{true}$, we consider the following: 
   for all $u \in B_t$, consider $
    h(u) = h'(u), \quad g(u) = g'(u), \quad p = p'$, and update 
    $\DP[t, g, h, p] := \text{true}$.


    



    

    

\medskip
\noindent\textbf{Join Node.}
Let $t$ be a join node with children $t_1$ and $t_2$
such that
\[
B_t = B_{t_1} = B_{t_2}.
\]
For every pair of valid entries
\[
\DP[t_1,g_1,h_1,p_1]=\text{true}
\quad\text{and}\quad
\DP[t_2,g_2,h_2,p_2]=\text{true},
\]
we combine the corresponding partial solutions at node~$t$,

we allow a transition only if the following consistency conditions hold for all $v \in B_t$:

\begin{itemize}
    \item If $h_1(v) = \text{center}$, then $h_2(v) \in \{\text{center}, \emptyset\}$.
    \item If $h_2(v) = \text{center}$, then $h_1(v) \in \{\text{center}, \emptyset\}$.
    \item If $h_1(v) = \text{leaf}$, then $h_2(v) \neq \text{center}$.
    \item If $h_2(v) = \text{leaf}$, then $h_1(v) \neq \text{center}$.
\end{itemize}




Moreover, combining the two partial solutions must not create an induced $P_4$ or a $C_3$. We define the new parameters as follows:
\begin{itemize}
    \item For each $v \in B_t$, we set
     $g(v) = \max\{g_1(v), g_2(v)\}$, and 
     
    \[h(v) =
    \begin{cases}
        \text{center}, & \text{if $h_1(v) = \text{center}$ or $h_2(v) = \text{center}$},\\
        \text{leaf}, & \text{if $h_1(v) = \text{leaf}$ or $h_2(v) = \text{leaf}$},\\
        \emptyset, & \text{otherwise},
    \end{cases}
    \]
    

    \item The number of edges is updated as
    
    \[p = p_1 + p_2 - 
    \begin{cases}
        1, & \text{if $|\{v \in B_t : h_1(v) \neq \emptyset \text{ and } h_2(v) \neq \emptyset\}| = 2$},\\
        0, & \text{otherwise}.
    \end{cases}
    \]
    
    This accounts for a solution edge whose both endpoints are in $B_t$, to avoid double counting.
\end{itemize}

Finally, we set 
$\DP[t, g, h, p] := \text{true}$.


\medskip
\noindent\textbf{Answer.}
Let $r$ be the root of the tree decomposition $T$. 
The size of a maximum edge open packing in $G$ is 
$\max \bigl\{\, p \;\big|\; \exists g, h : \DP[r, g, h, p] = \text{true} \,\bigr\}$. 
Here, $h$ encodes for each vertex in the root bag whether it participates in a solution edge (as a center or leaf) or is empty, and $g$ tracks adjacency to solution edges outside the bag.

The correctness of the algorithm follows directly from the description of the dynamic programming formulation. 
At each node of the nice tree decomposition, the DP table correctly encodes all feasible partial edge open packings in the subtree rooted at that node, with the functions $g$ and $h$ ensuring that no induced $P_4$ or $C_3$ is created and that edge endpoints are correctly tracked. 
By induction on the structure of the tree decomposition, the DP entries at the root correspond exactly to all feasible edge open packings in $G$. 

\medskip
\noindent\textbf{Running Time.} For a chordal graph $G$ it is known that the treewidth of $G$ equals the clique number of $G$, denoted by $\omega(G)$ or simply $\omega$. 
Let $\omega$ be the width of the tree decomposition, so $|B_t| \le \omega+1$ for all bags $B_t$. 
For each bag $B_t$, the DP table has dimensions:
$g: B_t \to \{0,1\}, h: B_t \to \{\text{center}, \text{leaf}, \emptyset\}, p \in [0, |E(G)|]$.  Thus, the number of DP entries per bag is at most 
$2^{|B_t|} \cdot 3^{|B_t|} \cdot (|E(G)|+1) = O(6^{\omega+1} \cdot |E(G)|)$. For each node, computing the DP table requires combining entries from at most two children (in the case of join nodes), and for introduce/forget nodes the update can be done in $O(6^{|B_t|})$ time per entry. 
Since the number of nodes in a nice tree decomposition is $O(n)$, the overall running time is 
$O(n \cdot 6^{2(\omega+1)} \cdot |E(G)|) = 2^{O(\omega)} \cdot \mathrm{poly}(n)$, 
i.e., the algorithm runs in FPT time parameterized by the clique number of the input chordal graph. This completes the proof of Theorem \ref{thm:chordal}.






\section{Conclusion}\label{EOP_Sec_5}
In this paper, we have designed polynomial-time algorithms to solve \maxeop~in distance-hereditary and biconvex bipartite graphs. We consider the decision version of the problem and show that the problem is FPT when parameterized by the clique number on chordal graphs. However, the complexity of the problem on chordal graphs in general remains unsettled and is an interesting direction to study.

\bibliographystyle{plain}
\bibliography{EOP_Bib}

\end{document}